\documentclass[12pt]{article}

\usepackage{amsthm, amsmath, amssymb, verbatim}
\usepackage{txfonts}

\numberwithin{equation}{section}
	     
\newtheorem{thm}{Theorem}[section] 
\newtheorem{prop}{Proposition}[section]
\newtheorem{lem}{Lemma}[section]

\theoremstyle{definition} 

\theoremstyle{remark} 
\newtheorem{rem}{Remark\rm}[section]

\title{Coincidence sets in quasilinear elliptic problems\\
of monostable type
}
\author{Shingo Takeuchi 
\thanks{This work was supported by KAKENHI (No. 20740094).}\\ 
Department of General Education, Kogakuin University\\
2665-1 Nakano, Hachioji, Tokyo  192-0015, 
JAPAN\\
E-mail: shingo@cc.kogakuin.ac.jp}
\date{}

\newcommand{\dist}{\operatorname{dist}}
\newcommand{\divg}{\operatorname{div}}
\newcommand{\ep}{\varepsilon}
\newcommand{\ol}{\overline}
\newcommand{\ul}{\underline}
\newcommand{\vp}{\varphi}
\newcommand{\Om}{\Omega}

\begin{document}
	
\maketitle 
     
\begin{abstract}
This paper concerns the formation of a coincidence set
for the positive solution of
the boundary value problem: $-\ep \Delta_p u=
 u^{q-1}f(a(x)-u)$ in $\Om$
with $u=0$ on $\partial \Om$,
where $\ep$ is a positive parameter, 
 $\Delta_p u=\divg(|\nabla u|^{p-2}\nabla u),\ 1<q \le p<\infty$, 
$f(s) \sim |s|^{\theta-1}s\ (s \to 0)$ for some $\theta>0$ and 
$a(x)$ is a positive smooth function
satisfying $\Delta_p a=0$ in $\Om$ with 
$\inf_\Om |\nabla a|>0$.
It is proved in this paper that 
if $0<\theta<1$
the coincidence set
$\mathcal{O}_\ep=\{x \in \Om:u_\ep(x)=a(x)\}$ 
has a positive measure and 
converges to $\Om$
with order $O(\ep^{1/p})$
as $\ep \to 0$.
Moreover, it is also shown that 
if $\theta \ge 1$, then $\mathcal{O}_\ep$ is empty for any $\ep>0$.
The proofs rely on comparison theorems and 
an energy method for obtaining local comparison functions.
\end{abstract}

 \section{Introduction} \setcounter{equation}{0}

 Let $\Om$ be a bounded domain in $\mathbb{R}^N\ (N \ge 2)$ with smooth
 boundary $\partial \Om$, and we consider the boundary value problem 
 of quasilinear elliptic equations of monostable type:
 \begin{equation}
  \label{eq:p}
  \begin{cases}
   -\ep \Delta_p u =u^{q-1}f(a(x)-u)
   \quad & \mbox{in}\ \Om,\\
   u \ge 0,\ u \not\equiv 0 & \mbox{in}\ \Om,\\
   u=0 & \mbox{on}\ \partial \Om,
  \end{cases}
 \end{equation}
 where $\ep$ is a positive parameter,
 $\Delta_p u$ denotes the $p$-Laplacian $\divg(\nabla_p u)$  
 with the $p$-gradient $\nabla_p u=|\nabla u|^{p-2}\nabla u$,
 $1<q \le p<\infty$, 
 $a:\Om \to \mathbb{R}$ is a 
 positive and smooth function 
 and 
 $f$ is a function satisfying the following
 conditions.

 (F1)\ $f \in C(\mathbb{R}) \cap C^1(\mathbb{R} \setminus \{0\})$ and $f(0)=0$.

 (F2)\ $f$ is strictly increasing on $\mathbb{R}$.

 (F3)\ There exists $\theta>0$ such that $\lim_{s \to 0}
 \frac{f(s)}{|s|^{\theta-1}s}=C$ for some $C>0$.

 By a solution of \eqref{eq:p} we mean a function $u \in 
 W^{1,p}_0(\Om)
 \cap L^{\infty}(\Om)$ satisfying \eqref{eq:p}
 (for details, see Section 2).
 Applying the theorem of D\'{i}az and Sa\'{a} \cite{DS}
 and the regularity result of Lieberman \cite{L},
 we see that if $\ep<\ep_a$ then 
  \eqref{eq:p} admits a unique positive solution 
  $u_{\ep} \in C^{1,\alpha}(\ol{\Om})$ for some 
  $\alpha \in (0,1)$;
  if $\ep \ge \ep_a$ then \eqref{eq:p} has no
  solution. 
  Here, $\ep_a=\infty$ if $p>q$ and $\ep_a=1/\lambda_{f(a)}$ 
  if $p=q$, where
 $\lambda_{f(a)}$ denotes the first eigenvalue of the definite weight
 eigenvalue problem
 \begin{equation*}
  \begin{cases}
   -\Delta_p u =\lambda f(a(x))|u|^{p-2}u 
   \quad & \mbox{in}\ \Om,\\
   u=0 & \mbox{on}\ \partial \Om,
  \end{cases}
 \end{equation*}
 and it can be characterized by 
 \begin{equation*}
  \lambda_{f(a)}=\inf_{u \in W^{1,p}_0(\Om),\ \neq 0}
   \frac{\displaystyle  \int_{\Om} |
   \nabla u(x)|^p\,dx}{\displaystyle \int_{\Om} f(a(x))
   |u(x)|^p\,dx}.
 \end{equation*}
 
 We define the \textit{coincidence set} of the positive solution $u_\ep$
 of \eqref{eq:p} with $a(x)$ as
 \begin{equation*}
  \mathcal{O}_\ep=\{x \in \Om:u_\ep(x)=a(x)\}.
 \end{equation*}

 In case $a(x)$ is constant, 
 problem \eqref{eq:p}
 has been already studied by several 
 authors. 
 Let $a(x) \equiv 1$ and $p=q>2$.
 Then, Guedda and V\'{e}ron \cite{GV} for $N=1$ and 
 Kamin and V\'{e}ron \cite{KV} for $N \ge 2$ 
 established that there exists
 a non-empty coincidence set $\mathcal{O}_\ep$ (or a \textit{flat core}, because 
 the graph of $u_\ep$ is flat on $\mathcal{O}_\ep$) 
 for $\ep$ small enough
 (when $\Om$ is a ball and $f(s)=s$, Kichenassamy and Smoller \cite{KS}
 had obtained the positive radial solution 
 with a flat core).
 They and 
 Garc\'{i}a-Meli\'{a}n and Sabina de Lis \cite{GS} 
 proved that if $0<\theta<p-1$, 
 then the flat core 
 has a positive measure for small $\ep \in (0,f(a)/\lambda_{f(a)})$ 
 and it converges to $\Om$ as $\dist(x,\mathcal{O}_\ep) \sim \ep^{1/p}\
 (\ep \to 0)$ for any $x \in \partial \Om$;
 while if $\theta \ge p-1$, then the flat core is empty. 
 These earlier results \cite{GS,GV,KV,KS} 
 are substantially sharpened by
 Guo \cite{Gu}. 
 Moreover, even if $a(x)$ is constant on a plural subdomain of $\Om$,
 there exists a flat core in each subdomain (see \cite{T1}). 
 General references for coincidence set are given in
 the monographs \cite{D} of D\'{i}az and \cite{PS} of
 Pucci and Serrin.

 In this paper we shall investigate
 the case where $a(x)$ is variable.
 It is heuristic that if the coincidence set 
 $\mathcal{O}_\ep$ has an interior point, then $a(x)$ has to satisfy
 $\Delta_p a=0$ on its neighborhood. 
 Inversely, we shall 
 assume $a(x)$ to be $p$-harmonic:
 $\Delta_p a=0$ in $\Om$,
and hence $a(x)$ satisfies the
equation of \eqref{eq:p}. 
Then, our major finding is that the $p$-harmonicity of $a(x)$
 is also a sufficient condition for an appearance of 
 coincidence set. 

Before stating the result, 
 we give precise conditions to $a(x)$:
 
 (A1) $\inf_{x \in \Om}a(x)>0$,
 
 (A2) $a \in C^{1,\alpha}(\ol{\Om})$
for some $\alpha \in (0,1)$ and $\Delta_p a=0$ in $\Om$, and 
%
%

 (A3) $\inf_{x \in \Om}|\nabla a(x)|>0$.

\noindent 
We notice that by DiBenedetto \cite{Di} and Tolksdorf \cite{To}, 
(A2) follows from, e.g.,  

(A2') there exists a domain 
      $\Om' \supset \ol{\Om}$ such that 
	$a \in W^{1,p}_{\rm loc}(\Om')$
	and $\Delta_pa=0$ in $\Om'$.
	
The following theorem suggests that with regard to the coincidence set of
 positive solution, it is unnecessary to assume
 $a(x)$ to be constant as in the past studies.
  
\begin{thm}
 \label{thm:main}
 Assume {\rm (A1)}, {\rm (A2)}
 and {\rm (A3)}. Let $0<\theta<1$.
 Then, there exist $L>0$ and $\ep_0 \in (0,\ep_a)$ 
 such that 
 for each $\ep \in (0,\ep_0)$
 the solution $u_{\ep}$ of \eqref{eq:p} satisfies
 $$u_{\ep}(x)=a(x)\quad \mbox{if}\ \dist(x,\partial \Om) \ge
	L{\ep}^{1/p}.$$
\end{thm}

The corresponding theorem for $p=2$ has been 
already proved in the author's paper
\cite{T2}. As mentioned above,
the condition $0<\theta<p-1$ seems to be
valid as a modification to the case $1<p<\infty$,
while the condition $0<\theta<1$ in the theorem
is same as that in case $p=2$.
However, this is natural because the principal
 part of equation of \eqref{eq:p} is neither degenerate nor
 singular in $\mathcal{O}_\ep$ when $a(x)$ satisfies the 
non-degeneracy condition (A3).
 
The condition $0<\theta<1$ in Theorem \ref{thm:main}
 is optimal in the following sense.

\begin{thm}
 \label{thm:main2}
 Assume $a(x)$ to be same in Theorem {\rm \ref{thm:main}}.
 Let $\theta \ge 1$.
 Then, for every $\ep \in (0,\ep_a)$, $u_\ep<a$ in $\Om$, 
and hence $\mathcal{O}_\ep =\emptyset$.
\end{thm}

In our approach, it is significant to study the translation
$-\ep \Delta_p(v-a)$ of the principal part
$-\ep \Delta_p v$.
Putting $\Phi_p(\nabla v,\nabla a)=\nabla_p(v-a)+\nabla_p a$
and using (A2), 
we see that $\Phi_p(0,\nabla a)=0$ and that
the translation can be represented as
the monotone operator $v \mapsto -\ep \divg 
\Phi_p(\nabla v,\nabla a)$. The vector-valued function 
$\Phi_p(\eta,\nabla a)$ has a different order at $\eta=0$ 
from what $\Phi_p(\eta,0)$  
has if and only if $a(x)$ is non-degenerate. 
This is the reason why
the conditions of $\theta$ in the theorems 
differ from those in case $a(x)$ is constant.

Theorems \ref{thm:main} and \ref{thm:main2} 
are proved in Section 4. In order to show Theorem
\ref{thm:main}, letting the solution $u_\ep$
be close to $a(x)$ as $\ep \to 0$ (the convergence will be shown in Section 2),  
we compare $u_\ep$ with a local comparison function which attains
$a(x)$. Such a comparison function is
 obtained in Section 3 by means of 
the energy method developed by D\'{i}az and V\'{e}ron \cite{DV}
(see also D\'{i}az \cite{D}, and Antontsev, D\'{i}az and Shmarev 
\cite{ADS}).
In proving Theorem \ref{thm:main2}, we give 
a Harnack type inequality by Trudinger \cite{Tr}
for an associated differential inequality.
Finally, in Section 5, we apply our method to the known 
case where $a(x)$ is constant and realize the necessity of modifying 
the condition of $\theta$ to $0<\theta<p-1$.

The corresponding theorems for $N=1$ to Theorems \ref{thm:main}
and \ref{thm:main2} have been already 
obtained in the author's paper \cite{T3}.


\begin{rem}
If $\Om=\mathbb{R}^N$, then the corresponding 
problem to \eqref{eq:p}
$$-\ep \Delta_p u=u^{q-1}f(a(x)-u) \quad \mbox{in}\ \mathbb{R}^N$$
is trivial.
Indeed, since $a(x)$ is a positive and $p$-harmonic function
in $\mathbb{R}^N$,
it is constant by Liouville's theorem for $p$-Laplacian 
\cite[Corollary 7.2.3]{PS} and any nonnegative 
solution of \eqref{eq:p}
must be the constant (see Du and Guo \cite{DG}).
\end{rem}

Through the paper, we denote by $C$ positive constants 
independent of $\ep$ and $\delta$, unless otherwise noted.

\section{Convergence to $a(x)$ as $\ep \to 0$}

In this section, we show that the solution of 
\eqref{eq:p} converges to $a(x)$ uniformly in any compact set of 
$\Om$ as $\ep \to 0$. 

A function $u=u_\ep \in W^{1,p}_0(\Om) \cap L^{\infty}(\Om)$ is called
a \textit{solution of} \eqref{eq:p} if $u \ge 0$ a.e. in $\Om$,
$u$ does not vanish in a set of positive measure, and
$$\ep \int_\Om \nabla_p u \cdot \nabla \vp \, dx
= \int_\Om u^{q-1}f(a(x)-u) \vp \,dx$$
for all $\vp \in W^{1,p}_0(\Om)$.
A function $u=u_\ep \in W^{1,p}_0(\Om) \cap L^{\infty}(\Om)$ is called
a \textit{supersolution} (resp. \textit{subsolution}) \textit{of} \eqref{eq:p} 
if $u \ge 0$ (resp. $u \le 0$) a.e. on $\partial \Om$ and
$$\ep \int_\Om \nabla_p u \cdot \nabla \vp \, dx
\ge\ (\mbox{resp.}\ \le )\ \int_\Om u^{q-1}f(a(x)-u) \vp \,dx$$
for all $\vp \in W^{1,p}_0(\Om)$ satisfying $\vp \ge 0$ a.e. in $\Om$.
If a function $u$ is not only a supersolution but also a subsolution,
then $u$ must be a solution of \eqref{eq:p}.

 We denote by $\lambda_1$ the first eigenvalue to the 
 following eigenvalue problem
 and by $z$ the corresponding eigenfunction to $\lambda_1$
 with $\|z\|_{L^\infty(\Om)}=\sup_{x \in \Om}{|z(x)|}=1$:
 \begin{equation*}
  \begin{cases}
   -\Delta_p z =\lambda |z|^{p-2}z
   \quad & \mbox{in}\ \Om,\\
   z=0 & \mbox{on}\ \partial \Om.
  \end{cases}
 \end{equation*}
 It is well-known that $\lambda_1>0,\ z \in C^{1,\alpha}(\ol{\Om})$ 
for some $\alpha \in (0,1)$ and $z>0$ in $\Om$.
 Let $B(x_0,r)=\{x \in \mathbb{R}^N: |x-x_0|<r\}$,
 $\Om_{\ep}=\{x \in \Om:\dist(x,\partial \Om) \ge \ep\}$
 and $d=\inf_{x \in \Om}{a(x)}/2>0$.

\begin{prop}
 \label{lem:convergence}
 Assume $a(x)$ to satisfy {\rm (A1)} and {\rm (A2)}.
 For each $\delta \in (0,2d)$,
 there exist $K>0$ and $\ep_* \in (0,\ep_a)$ 
 such that if $\ep \in (0,\ep_*)$ then the 
 solution $u_{\ep}$ of \eqref{eq:p} satisfies
 $$a(x)-\delta \le u_{\ep}(x) \leq a(x)\quad 
 \mbox{for all $x \in \Om_{K\ep^{1/p}}$}.
 $$
\end{prop}

\begin{proof}
 It is clear from (A2) that $\ol{u}=a$ is a
 supersolution of \eqref{eq:p} for every $\ep>0$.

 We shall construct a subsolution of \eqref{eq:p}.
 From the uniform continuity of $a(x)$
 in $\ol{\Om}$, 
 there exists $r>0$ such that for every $x_0 \in \Om$,
 $a(x)>a(x_0)-\delta/2$ for all $x \in B(x_0,r) \cap \Om$,
 and hence for each $x \in B(x_0,r) \cap \Om$,
 $a(x)-u>\delta/2$ for all $u \in [0,a(x_0)-\delta]$.
 Therefore, $f(a(x)-u) \ge \sigma=f(\delta/2)$
 for all $x \in B(x_0,r)\cap \Om$ and $u \in [0,a(x_0)-\delta]$.
 Let $K>0$ be a constant satisfying
 $K^p > \lambda_1\|a\|_{L^\infty(\Om)}^{p-q}/\sigma$ 
 and choose $\ep_* \in (0,\ep_a)$
 such that
 $K\ep_*^{1/p}<r$.

 Take any $\ep \in (0,\ep_*)$ and $x_0 \in 
 \Om_{K\ep^{1/p}}$.
 Changing scaling as $\ul{z}(x)=z((x-x_0)/(K\ep^{1/p}))$, 
 we have
 \begin{equation*}
  \begin{cases}
   -\ep \Delta_p \ul{z}=\dfrac{\lambda_1}{K^p}\ul{z}^{p-1}
   \quad & \mbox{in}\ B(x_0,K\ep^{1/p}),\\
   \ul{z}=0 & \mbox{on}\ \partial B(x_0,K\ep^{1/p}).
  \end{cases}
 \end{equation*}
 Then the function
 \begin{equation*}
  \ul{u}(x)=\begin{cases}
	     (a(x_0)-\delta)
	     \ul{z}(x), \quad & x \in B(x_0,K\ep^{1/p}),\\
	     0, & x \in \Om \setminus B(x_0,K\ep^{1/p})
	    \end{cases}
 \end{equation*}
 is a nonnegative subsolution of \eqref{eq:p}.
  Indeed, $a(x_0) \ge 2d>\delta$, and 
 for every $\vp \in W^{1,p}_0(\Om)$ with $\vp \ge 0$
 \begin{align*}
  \frac{1}{(a(x_0)-\delta)^{q-1}} &
  \left(\ep \int_{\Om} \nabla_p \ul{u} \cdot
  \nabla \vp\, dx
  -\int_{\Om} \ul{u}^{q-1}f(a(x)-\ul{u})\vp\, dx\right)\\
  & \le - \ep \int_{B(x_0,K\ep^{1/p})} (a(x_0)-\delta)^{p-q}
  \Delta_p \ul{z}\, \vp\,dx
  - \sigma\int_{B(x_0,K\ep^{1/p})}
  \ul{z}^{q-1}\vp\, dx\\
  &= \int_{B(x_0,K\ep^{1/p})}
  \left(\frac{\lambda_1(a(x_0)-\delta)^{p-q}}{K^p}
  \ul{z}^{p-q}-\sigma\right)
  \ul{z}^{q-1}\vp\, dx\\
  & \le \left(\frac{\lambda_1\|a\|_{L^\infty(\Om)}^{p-q}}{K^p}-\sigma\right)
  \int_{B(x_0,K\ep^{1/p})}\ul{z}^{q-1}\vp\, dx \le 0.
 \end{align*}

 Since 
 $\ul{u}<\ol{u}$ in $\Om$, there exists 
 a solution $u^*$ of \eqref{eq:p} with 
 $\ul{u} \le u^* \le \ol{u}$ in $\Om$ 
 (e.g., Deuel and Hess \cite{DH}).
 As mentioned in Section 1, the solution
 of \eqref{eq:p} is unique. Therefore, $u^*=u_\ep$, and hence 
 $\ul{u} \le u_\ep \le \ol{u}$ in $\Om$. 
 In particular, $a(x_0)-\delta \le u_\ep(x_0) \leq
 a(x_0)$ for all $x_0 \in \Om_{K\ep^{1/p}}$
 when $0<\ep<\ep_*$.
\end{proof}

\begin{rem}
 Even if (A2) is not assumed, then we can 
 prove that $|u_\ep-a|<\delta$. Indeed, we can construct a 
 supersolution of \eqref{eq:p} close to $a(x)$ from above.
 Let $p \ge 2$ for simplicity, 
 and assume $\ol{u}$ to be an arbitrary \textit{smooth}
 function satisfying $a+\delta/2 < \ol{u} < a+\delta$.
 Since
 \begin{align*}
  -\ep \Delta_p\ol{u}-\ol{u}^{q-1}f(a(x)-\ol{u})
  \ge -\ep \Delta_p \ol{u}+C (\ol{u}-a(x))^\theta
  \ge -\ep \Delta_p \ol{u}+C\left(\frac{\delta}{2}\right)^\theta
 \end{align*}
 for all $x \in \Om$ and $\Delta_p\ol{u}$ is continuous in $\ol{\Om}$,
 the last expression can be positive 
 provided $\ep$  is small enough. For the case $1<p<2$, 
 we refer to \cite{T1}.
\end{rem}

\section{Auxiliary problem near $a(x)$}

In this section, we show that there exists a comparison function
with dead core, which satisfies an equation having a
subsolution $a-u_\ep \ge 0$.

We define the vector-valued function $\Phi_p: \mathbb{R}^N 
\times \mathbb{R}^N \to \mathbb{R}^N$ as
$$\Phi_p(\eta,\xi)=|\eta-\xi|^{p-2}(\eta-\xi)+|\xi|^{p-2}\xi.$$
In particular, we note 
that $\Phi_p(\nabla u,\nabla v)=\nabla_p(u-v)+\nabla_pv$ for gradients.

The following lemma
means that for each $\xi \neq 0$ the function $\Phi_p(\eta,\xi)$ 
is of order $1$ at $\eta=0$.
\begin{lem}
 \label{lem:order}
For all $\eta,\ \xi \in \mathbb{R}^N$
 with $|\eta-\xi|+|\xi|>0$
 \begin{align}
  \label{eq:ge}
  \Phi_p(\eta,\xi)\cdot\eta & \ge \min\{p-1,2^{2-p}\}
  (|\eta-\xi|+|\xi|)^{p-2}|\eta|^2,\\
  \label{eq:le}
  |\Phi_p(\eta,\xi)| & \le \max\{p-1,2^{2-p}\}
  (|\eta-\xi|+|\xi|)^{p-2}|\eta|.
 \end{align}
 For all $\eta,\ \eta',\ \xi \in \mathbb{R}^N$ 
with $|\eta-\xi|+|\eta'-\xi|>0$
  \begin{align}
  \label{eq:sage}
  (\Phi_p(\eta,\xi)-\Phi_p(\eta',\xi))\cdot(\eta-\eta') 
  & \ge  \min\{p-1,2^{2-p}\}(|\eta-\xi|
  +|\eta'-\xi|)^{p-2}|\eta-\eta'|^2,\\
  \label{eq:sale}
  |\Phi_p(\eta,\xi)-\Phi_p(\eta',\xi)| 
  & \le \max\{p-1,2^{2-p}\}
  (|\eta-\xi|+|\eta'-\xi|)^{p-2}|\eta-\eta'|. 
 \end{align}
\end{lem}

\begin{proof}
 By the mean value theorem, we have
 \begin{align}
  \label{eq:seki}
  (\Phi_p(\eta,\xi),\eta)&=(p-1)|\eta|^2\int_0^1
  |t\eta-\xi|^{p-2}\,dt,\\
  \label{eq:abs}
  |\Phi_p(\eta,\xi)|&=(p-1)|\eta|\int_0^1|t\eta-\xi|^{p-2}\,dt.
 \end{align}
 Since $|t\eta-\xi|=|t(\eta-\xi)-(1-t)\xi|
 \le |\eta-\xi|+|\xi|$ for all
 $t \in [0,1]$, equation \eqref{eq:seki} yields
 \eqref{eq:ge} if $1<p \le 2$, while
 \eqref{eq:abs} yields \eqref{eq:le} if $p \ge 2$.

 Putting $t_0=|\xi|/(|\eta-\xi|+|\xi|) 
 \in (0,1]$, we have
 \begin{align*}
  |t\eta-\xi|
  \ge|t|\eta-\xi|-(1-t)|\xi||
  =(|\eta-\xi|+|\xi|)|t-t_0|.
 \end{align*}
  If $p>2$ (resp. $1<p<2$), then for every $t_0 \in (0,1]$ we have that 
 $\int_0^1|t-t_0|^{p-2}\,dt \ge$ (resp. $\le$) 
 $2 \int_0^{1/2}z^{p-2}\,dz
 =2^{2-p}/(p-1)$, thus \eqref{eq:seki} 
 (resp.\ \eqref{eq:abs}) yields \eqref{eq:ge}
 (resp.\ \eqref{eq:le}).
 
Since $\Phi_p(\eta,\xi)-\Phi_p(\eta',\xi)
=\Phi_p(\eta-\eta',\xi-\eta')$,
\eqref{eq:sage} and \eqref{eq:sale}
follow from \eqref{eq:ge} and \eqref{eq:le}, respectively. 
\end{proof}

Let $\Lambda$ be a positive constant.
Take $x_0 \in \Om,\ \delta \in (0,1)$ 
and $\ep \in (0,1)$ such that
$B=B(x_0,\ep^{1/p}) \subset \Om$.
Consider the boundary value problem
  \begin{equation}
   \label{eq:lcf}
   \begin{cases}
    -\ep \divg\Phi_p(\nabla w,\nabla a)+\Lambda |w|^{\theta-1}w=0 
    \quad & \mbox{in}\ B,\\
    w=\delta & \mbox{on}\ \partial B.
   \end{cases}
  \end{equation}
 For Propositions \ref{prop:comparison} and 
\ref{prop:existence} below,
we assume only $a \in W^{1,p}(B)$ 
without (A1), (A2) and (A3).

\begin{prop}
\label{prop:comparison}
Let $g$ be a non-decreasing function, and
suppose that $u,\ v \in W^{1,p}(B) \cap L^\sigma(B)$, where
$\sigma \in [1,\infty]$, satisfy
$g(u),\ g(v) \in L^{\sigma^*}(B)$, where
$\sigma^*=\frac{\sigma}{\sigma-1}$
$(\sigma^*=\infty$ if $\sigma=1$ and $\sigma^*=1$ if $\sigma=\infty)$,
and
\begin{equation*}
\begin{cases}
-\divg\Phi_p(\nabla u,\nabla a)+g(u) 
\le -\divg\Phi_p(\nabla v,\nabla a)+g(v) 
\quad & \mbox{in}\ B,\\
u \le v & \mbox{on}\ \partial B.
\end{cases}
\end{equation*}
Then, $u \le v$ a.e. in $B$. 
\end{prop}

\begin{proof}
Using $(u-v)^+ \in W^{1,p}_0(B) \cap L^\sigma(B)$ as a test function, 
we get
$$\int_D
(\Phi_p(\nabla u,\nabla a)-\Phi_p(\nabla v,\nabla a))\cdot
(\nabla u-\nabla v)\,dx
\le -\int_D (g(u)-g(v))
(u-v)\,dx \le 0,$$
where $D=\{x \in B: u(x)>v(x)\}$.
On the other hand,  
the integrand of the left-hand side is non-negative
because of \eqref{eq:sage}. 
Thus, we conclude $\nabla u=\nabla v$ a.e. in $D$, 
and hence $\nabla (u-v)^+=0$ a.e. in $B$,
which means $(u-v)^+=0$ a.e. in $B$. 
Therefore, $u \le v$ a.e. in $B$. 
\end{proof}
  
\begin{prop}
\label{prop:existence}
For any $\ep>0$, 
there exists a unique solution $w \in W^{1,p}(B) \cap L^{\infty}(B)$ 
of \eqref{eq:lcf}. Moreover, $0 \le w \le \delta$
a.e. in $B$.
\end{prop}

\begin{proof}
We set the $C^1$-energy functional $J$ corresponding 
to \eqref{eq:lcf} as
$$J(u)=\frac{\ep}{p}\int_B |\nabla u-\nabla a|^p\,dx
+\ep\int_B \nabla_p a \cdot \nabla u\,dx
+\Lambda \int_B |u|^{1+\theta}\,dx,$$
which is defined in
$$K=\{u \in W^{1,p}(B) \cap L^{1+\theta}(B):u-\delta \in W^{1,p}_0(B)\}.$$
Since
$$|\nabla_p a\cdot \nabla u| \le |\nabla a|^{p-1}|\nabla u-\nabla a|+|\nabla a|^p
\le \frac{1}{2p}|\nabla u-\nabla a|^p+C |\nabla a|^p,$$
we have
\begin{equation}
\label{eq:J}
J(u) \ge \frac{\ep}{2p}\int_B|\nabla u-\nabla a|^p\,dx+\Lambda
\int_B|u|^{1+\theta}\,dx-C \ep \int_B|\nabla a|^p\,dx.
\end{equation}
Then we see that 
$J$ is bounded from below and $J_0=\inf_{u \in K}J(u)$ exists. 
It suffices to show that there exists $w \in K$ such that 
$J(w)=J_0$. 

Let $\{u_n\}$ be a minimizing sequence such that $u_n \in K$
and $J(u_n) \to J_0$ as $n \to \infty$. Then, by \eqref{eq:J}
we obtain
\begin{align*}
\int_B|\nabla u_n-\nabla a|^p\,dx, \quad 
\int_B|u_n|^{1+\theta}\,dx \quad \le C,
\end{align*}
so that $\{u_n-\delta\}$ and $\{u_n\}$ are bounded in 
the reflexive Banach spaces $W^{1,p}_0(B)$ and 
$L^{1+\theta}(B)$, respectively.
Thus, we can choice a subsequence, 
which is denoted $u_n$ again, and 
$w \in K$
such that $u_n \to w$ weakly in $W^{1,p}(B)$
and weakly in $L^{1+\theta}(B)$. Thus, 
\begin{align}
& \liminf_{n \to \infty}\|u_n-a\|_{W^{1,p}(B)}
\ge \|w-a\|_{W^{1,p}(B)}, \label{eq:conv1}\\
& \lim_{n \to \infty}\int_B \nabla_p a \cdot \nabla u_n\,dx 
=\int_B \nabla_p a \cdot \nabla w\,dx, \label{eq:conv2}\\
& \liminf_{n \to \infty}\|u_n\|_{L^{1+\theta}(B)}
\ge \|w\|_{L^{1+\theta}(B)}. \label{eq:conv3}
\end{align}
Since 
$u_n \to w$ strongly in $L^p(B)$
by the Poincar\'{e} inequality,
it follows from \eqref{eq:conv1}
that
\begin{equation}
\label{eq:conv4}
\liminf_{n \to \infty} \|\nabla (u_n-a)\|_{L^p(B)}
\ge \|\nabla (w-a)\|_{L^p(B)}. 
\end{equation}
Therefore, by \eqref{eq:conv2}, \eqref{eq:conv3} and 
\eqref{eq:conv4}, we conclude that 
$J_0=\liminf_{n \to \infty}J(u_n) \ge J(w) \ge J_0$,
so that $J(w)=J_0$. The uniqueness and the boundedness 
of solutions follow from 
Proposition \ref{prop:comparison}
with $g(s)=|s|^{\theta-1}s$ and $\sigma=1+\theta$.
\end{proof}

To show that the solution $w$ of \eqref{eq:lcf}
has a dead core for any $\ep>0$, scaling is useful:   
setting $y=\ep^{-1/p}(x-x_0),\ \tilde{w}(y)=\tilde{w}(y;\ep,x_0)=w(x+\ep^{1/p}y)$
and $\tilde{a}(y)=\tilde{a}(y;\ep,x_0)=a(x_0+\ep^{1/p}y)$
in \eqref{eq:lcf}, we have
  \begin{equation}
   \label{eq:lcftw}
   \begin{cases}
    -\divg\Phi_p(\nabla \tilde{w},\nabla \tilde{a})+
\Lambda \tilde{w}^{\theta}=0
    \quad & \mbox{in}\ B(0,1),\\
    \tilde{w}=\delta & \mbox{on}\ \partial B(0,1).
   \end{cases}
  \end{equation}
We shall write $B_\rho$ to represent $B(0,\rho)$.

\begin{lem}
 \label{lem:w_x}
 Let $a(x)$ satisfy {\rm (A2)}, and
 assume $\tilde{w}$ to be the unique solution of \eqref{eq:lcftw}. 
	Then $\tilde{w} \in C^{1,\alpha}(\ol{B_1})$ for some 
$\alpha \in (0,1)$ and $\|\nabla(\tilde{w}-\tilde{a})\|_{L^{\infty}
(B_1)} \le C$, 
 where $C$ is independent of $\ep,\ \delta$ and $x_0$.
\end{lem}

\begin{proof}
Setting 
$v(y)=\tilde{w}(y)-\tilde{a}(y)$,
we have
  \begin{equation*}
   \begin{cases}
    -\Delta_p v+\Lambda (v+\tilde{a})^{\theta}=0 
    \quad & \mbox{in}\ B_1,\\
    v=\delta+\tilde{a} & \mbox{on}\ \partial B_1.
   \end{cases}
  \end{equation*}
Since $\|v+\tilde{a}\|_{L^{\infty}(B_1)} \le \delta \le 1$ 
by Proposition \ref{prop:comparison} and 
$\delta+\tilde{a}\ |_{\partial B_1} \in 
C^{1,\alpha}(\partial B_1)$ with $\|\delta+\tilde{a}\|_{C^{1,\alpha}(\partial B_1)}
\le \|\delta+\tilde{a}\|_{C^{1,\alpha}(\ol{B_1})}
\le 1+\|a\|_{C^{1,\alpha}(\ol{\Om})}$
(for the norm of $C^{1,\alpha}(\partial B_1)$, see 
Gilbarg and Trudinger \cite[Section 6.2]{GT}), 
it follows from a regularity result of Lieberman \cite{L}
that $v \in C^{1,\alpha}(\ol{B_1})$ and 
$\|v\|_{C^{1,\alpha}(\ol{B_1})} \le C$ 
for some $\alpha \in (0,1)$ and $C>0$ are 
independent of $\ep,\ \delta$ and $x_0$.
In particular, $\|\nabla v\|_{L^\infty(B_1)} \le C$.
\end{proof}

 \begin{prop}
\label{prop:lcf}
  Let $a(x)$ satisfy {\rm (A2)} and {\rm (A3)}, 
  and assume 
  $w$ to be the unique solution of \eqref{eq:lcf}.
  If $0<\theta<1$,
  then there exists $M>0$ independent of $\ep,\ \delta$ and $x_0$ 
  such that
  $w(x)=0$ for all $x \in 
  B(x_0,(1-M\delta^{(1+\theta)\gamma})^{1/\tau}\ep^{1/p})$,
  where
  \begin{align*}
  \gamma&=\frac{\frac{1}{1+\theta}-\frac12}
	{N\left(\frac{1}{1+\theta}-\frac12\right)+1}
	\in \left(0,\frac{1}{N+2}\right),\\
  \tau&=2N\left(\frac{1}{1+\theta}-\frac12\right)+2
	\in \left(2,N+2\right).
  \end{align*}
  In particular, $w(x_0)=0$ for arbitrary $\ep>0$ 
if $\delta^{(1+\theta)\gamma}<M^{-1}$.
 \end{prop}

\begin{proof}
 It is sufficient to prove
 the existence of dead core for the solution of \eqref{eq:lcftw}.
  To do this, we follow 
 the energy method developed by D\'{i}az and V\'{e}ron \cite{DV}
 (see also D\'{i}az \cite{D}, and Antontsev, D\'{i}az and
 Shmarev \cite{ADS}).

 We define the diffusion
 and absorption energy functions $E_D(\rho)$ and $E_A(\rho)$ 
 in $(0,1)$
 as follows:
 \begin{align*}
  E_D(\rho)&= \int_{B_\rho} 
  \Phi_p(\nabla \tilde{w}(y),\nabla \tilde{a}(y))\cdot \nabla \tilde{w}(y)\,dy,\\
  E_A(\rho)&=\int_{B_\rho}|\tilde{w}(y)|^{1+\theta}\,dy.
 \end{align*}
 The total energy function $E_T(\rho)$ is defined as
 \begin{equation*}
  E_T(\rho)=E_D(\rho)+\Lambda E_A(\rho).
 \end{equation*}
 The global total energy $E_T(1)$ is finite.
 Indeed, (we write $w,\ a$ instead 
 of $\tilde{w},\ \tilde{a}$,
 respectively),
 multiplying the equation of \eqref{eq:lcftw} by 
 the nonnegative function $\delta-w 
 \in W^{1,p}_0(B_1)$ and 
 integrating by parts in $B_1$, we have 
 \begin{equation}
  \label{eq:total}
   E_T(1) \le \Lambda\delta^{1+\theta}|B_1| 
   \le C\delta^{1+\theta}.
 \end{equation}

  Multiplying the equation of \eqref{eq:lcftw} by $w$ and 
 integrating by parts in $B_\rho$, we have also
 (now we shall write $S_\rho$ to represent $\partial B_\rho$)
 \begin{equation}
  \label{eq:E_T1}
  E_T(\rho)=\int_{S_\rho}\Phi_p(\nabla w(y),\nabla a(y))
\cdot n\,w(y)\,ds,
 \end{equation}
 where $n=n(s)$ is the outward normal vector at $y \in S_\rho$.
 By \eqref{eq:E_T1}, Lemmas \ref{lem:order} and \ref{lem:w_x}
 with (A3)
 \begin{align}
  \label{eq:E_T}
  E_T(\rho) &=\int_{S_\rho}|\Phi_p(\nabla w,\nabla a)||w|\,ds \notag \\
  & \le \left( \int_{S_\rho}|\Phi_p(\nabla w,\nabla a)|^2\,ds\right)^{1/2}
  \left( \int_{S_\rho}|w|^2\,ds\right)^{1/2} \notag \\
  & \le \left( \int_{S_\rho}(|\nabla w-\nabla a|+|\nabla a|)^{2(p-2)}
	(\Phi_p(\nabla w,\nabla a)\cdot 
  \nabla w )\,ds\right)^{1/2}\|w\|_{L^2(S_\rho)}
	\notag \\
  & \le C \left( \int_{S_\rho}\Phi_p(\nabla w,\nabla a)
  \cdot \nabla w\,ds\right)
  ^{1/2}\|w\|_{L^2(S_\rho)}.
 \end{align}
 On the other hand, by using spherical coordinates $(\omega,r)$ with
 center $x_0$, we have
 \begin{equation*}
  E_D(\rho)=\int_0^\rho \int_{S^{N-1}} \Phi_p(\nabla w(r\omega),
  \nabla a(r\omega))\cdot
   \nabla w(r\omega)\, r^{N-1}\,d\omega\,dr.
 \end{equation*}
 Hence, $E_D$ is almost everywhere differentiable and
 \begin{align}
  \label{eq:dE_D}
  \frac{dE_D(\rho)}{d\rho}
   & =\int_{S^{N-1}}\Phi_p(\nabla w(\rho \omega),
   \nabla a(\rho \omega))\cdot \nabla w(\rho
   \omega)\, \rho^{N-1}\,d\omega \notag \\
   & =\int_{S_\rho} \Phi_p(\nabla w,\nabla a)\cdot \nabla w
   \,ds.
 \end{align}
 Similarly,
\begin{align}
  \label{eq:dE_A}
  \frac{dE_A(\rho)}{d\rho}
   =\int_{S_\rho} |w|^{1+\theta}\,ds.
 \end{align}
 Moreover, since $0<\theta<1$, 
 we have the following inequality (see 
 D\'{i}az \textit{et al}. \cite{DV,D,ADS}):
\begin{equation*}
  \|w\|_{L^2(S_\rho)} \le C\left(\|\nabla w\|_{L^2(B_\rho)}
  +\rho^{-\alpha}\|w\|_{L^{1+\theta}(B_\rho)}\right)^{\beta}
  \|w\|_{L^{1+\theta}(B_\rho)}^{1-\beta},
\end{equation*}
 where $C=C(N,\theta)$ and 
 \begin{align*}
  \alpha&=\frac{N(1-\theta)+2(1+\theta)}{2(1+\theta)}
	=N\left(\frac{1}{1+\theta}-\frac12\right)+1 
	\in \left(1,\frac{N}{2}+1\right) \subset (1,\infty),\\ 
   \beta&=\frac{N(1-\theta)+1+\theta}{N(1-\theta)+2(1+\theta)}
   =\frac{N\left(\frac{1}{1+\theta}-\frac12\right)+\frac12}
   {N\left(\frac{1}{1+\theta}-\frac12\right)+1} \in 
	\left(\frac12,\frac{N+1}{N+2}\right) \subset (0,1).
 \end{align*}
 Thus, from \eqref{eq:ge} and Lemma \ref{lem:w_x}, we 
obtain $E_D(\rho) \ge C\|\nabla w\|_{L^2(B_\rho)}^2$, so that 
 \begin{align}
  \label{eq:E_A}
  \|w\|_{L^2(S_\rho)}^{1/\beta}
  &\le C\left(\|\nabla w\|_{L^2(B_\rho)}+\rho^{-\alpha}
  \|w\|_{L^{1+\theta}(B_\rho)}\right)\|w\|_{L^{1+\theta}(B_\rho)}
  ^{\frac{1-\beta}{\beta}} \notag \\
  &= C\left(\|\nabla
  w\|_{L^2(B_\rho)}\|w\|_{L^{1+\theta}(B_\rho)}
  ^{\frac{1-\beta}{\beta}}
  +\rho^{-\alpha}\|w\|_{L^{1+\theta}(B_\rho)}^{1/\beta}\right)
   \notag\\
  & \le C\rho^{-\alpha}\left(\rho^{\alpha}
  E_D(\rho)^{\frac12}
  E_A(\rho)^{\frac{1-\beta}{\beta(1+\theta)}}
  +E_A(\rho)^{\frac{1}{\beta (1+\theta)}}\right)  \notag\\
  & \le C\rho^{-\alpha}\left(
  E_T(\rho)^{\frac12+\frac{1-\beta}{\beta(1+\theta)}}
  +E_A(1)^{\frac{1}{1+\theta}-\frac12}
  E_A(\rho)^{\frac12+\frac{1-\beta}{\beta(1+\theta)}}
  \right)  \notag\\
  & \le C\rho^{-\alpha} 
  E_T(\rho)^{\frac12+\frac{1-\beta}{\beta(1+\theta)}}.
 \end{align}
 Here we have used 
 that $E_A(1) \le C \delta^{1+\theta}<C$ and $0<\theta<1$.
 Combining \eqref{eq:E_T}--\eqref{eq:dE_A} and \eqref{eq:E_A},
 we obtain
 \begin{align*}
  E_T(\rho) & \le  C 
  \left(\frac{dE_T(\rho)}{d\rho}\right)^{1/2}
  \rho^{-\alpha \beta}
  E_T(\rho)^{\frac{\beta}{2}+\frac{1-\beta}{1+\theta}},
 \end{align*}
 that is,
 \begin{equation*}
   \frac{dE_T(\rho)}{d\rho} 
   \ge C \rho^{\tau-1}E_T(\rho)^{1-\gamma},
 \end{equation*}
 where
 \begin{align*}
  \gamma&=
2(1-\beta)\left(\frac{1}{1+\theta}-\frac12\right)
	=\frac{\frac{1}{1+\theta}-\frac12}
	{N\left(\frac{1}{1+\theta}-\frac12\right)+1}
	\in \left(0,\frac{1}{N+2}\right),\\
  \tau&=1+2\alpha\beta=2N \left(\frac{1}{1+\theta}-\frac12\right)+2
	\in (2,N+2).
  \end{align*}
 Integrating it on $[\rho,1]$ 
 and using \eqref{eq:total}, we have
 \begin{align*}
  E_T(\rho)^{\gamma} & \le E_T(1)^{\gamma}-
  C(1-\rho^{\tau})
  \le C(\rho^{\tau}-
  (1-M\delta^{(1+\theta) \gamma}))
 \end{align*}
 for some $M>0$,
 thus $E_T((1-M\delta^{(1+\theta) \gamma})^{1/\tau})=0$, i.e., 
 $\tilde{w}(y)=0$ for all $y \in 
 B(0,(1-M\delta^{(1+\theta) \gamma})^{1/\tau})$. 
 Scaling back to $x$, we conclude the assertion.
\end{proof}

\section{Proofs of Theorems}

Now we are in a position to prove Theorems \ref{thm:main} and
\ref{thm:main2}.

\begin{proof}[Proof of Theorem $\ref{thm:main}$]
Fix $\delta \in (0,d)$ such that $M\delta^{(1+\theta)\gamma}<1$,
where $M$ and $\gamma$ are the constants appearing in Proposition 
\ref{prop:lcf}.
 Thanks to the $p$-harmonicity of $a(x)$,
 the function $v=a-u_\ep$ satisfies that
 $-\ep \divg \Phi_p(\nabla v,\nabla a)=-(a(x)-v)^{q-1}f(v)$
 in the distribution sense in $\Om$. Since
 \begin{align*}
   (a(x)-s)^{q-1}f(s) \ge d^{q-1} C s^\theta=:\Lambda_1 s^{\theta}
  & \quad \mbox{for all $x \in \Om$
  and $s \in [0,\delta]$} 
 \end{align*}
 and
 by Proposition \ref{lem:convergence}, 
 $\max_{x \in \Om_{K\ep^{1/p}}} v_{\ep}(x) \le \delta$
 for every $\ep \in (0,\ep_*)$,
 we have
 \begin{equation}
  \label{eq:pm}
   -\ep \divg \Phi_p(\nabla v,\nabla a)+\Lambda_1 v^{\theta} \le 0
   \quad \mbox{in}\ \Om_{K\ep^{1/p}}.
 \end{equation}

 Let $\ep_0 \in (0,\ep_*)$ be small such that
 $\Om_{(K+1)\ep_0^{1/p}} \neq \emptyset$. 
 Take any $\ep \in (0,\ep_0)$ and $x_0 \in \Om_{(K+1)\ep^{1/p}}$.
 Letting $w$ be the solution of \eqref{eq:lcf},
 we can see
 \begin{equation}
  \label{eq:h}
   \begin{cases}
    -\ep \divg \Phi_p(\nabla w,\nabla a)+
	\Lambda_1 w^\theta=0 \quad & \mbox{in}\ 
    B(x_0,\ep^{1/p}),\\
    w=\delta & \mbox{on}\ \partial B(x_0,\ep^{1/p}).
   \end{cases}
 \end{equation}
 Since $B(x_0,\ep^{1/p}) \subset \Om_{K\ep^{1/p}}$ and 
 $v \le \delta=w$ on $\partial B(x_0,\ep^{1/p})$,
 it follows from 
 \eqref{eq:pm} and 
 \eqref{eq:h} that $v$ is a subsolution of \eqref{eq:h}.
 Therefore, Proposition \ref{prop:comparison} gives 
 $v \le w$ in $B(x_0,\ep^{1/p})$.
 Proposition \ref{prop:lcf} implies that
 $0 \le v_{\ep}(x_0) \le w(x_0)=0$, and hence
 $u(x_0)=a(x_0)$ for all $x_0 \in \Om_{(K+1)\ep^{1/p}}$. 
 This completes the proof of 
 Theorem \ref{thm:main}.
\end{proof}

\begin{proof}[Proof of Theorem $\ref{thm:main2}$]
Let $u_\ep$ be a solution of \eqref{eq:p}.
The function $v=a-u_\ep \ge 0,\ \not\equiv 0$, satisfies
\begin{equation*}
  -\ep \divg \Phi_p(\nabla v,\nabla a)+ \Lambda_2 v^{\theta} \ge 0
\end{equation*}
for some $\Lambda_2>0$. 
Since $u_\ep \in C^1(\ol{\Om})$ by the regularity result of 
Lieberman \cite{L}, so is $v$, and there exists $k>0$ such that 
$\|\nabla v\|_{L^\infty(\Om)} \le k$. We define
\begin{align*}
M_{p,k}&=\sup_{|\eta| \le k,x \in \Om}(|\eta-\nabla a(x)|
+|\nabla a(x)|)^{p-2},\\
m_{p,k}&=\inf_{|\eta| \le k,x \in \Om}(|\eta-\nabla a(x)|
+|\nabla a(x)|)^{p-2},
\end{align*}
which are both finite and positive for any $p>1$
because of (A3).
Then, $v$ is also a nonnegative bounded function satisfying
\begin{equation*}
  -\ep \divg \tilde{\Phi}_p(\nabla v,\nabla a)+ 
\Lambda_2 |v|^{\theta-1}v \ge 0,
\end{equation*}
where $\tilde{\Phi}_p(\eta,\nabla a)$ is a vector measurable
function as 
\begin{equation*}
\tilde{\Phi}_p(\eta,\nabla a)=
\begin{cases}
\Phi_p(\eta,\nabla a) & \mbox{if}\ |\eta|\le k,\\
M_{p,k} \eta & \mbox{if}\ |\eta|>k,
\end{cases}
\end{equation*}
which satisfies (from \eqref{eq:le} and \eqref{eq:ge}
in Lemma \ref{lem:order})
\begin{align*}
|\tilde{\Phi}_p(\eta,\nabla a(x))|
& \le M_{p,k} \max\{p-1,2^{2-p}\}\,|\eta|,\\
\tilde{\Phi}_p(\eta,\nabla a(x))\cdot \eta
& \ge m_{p,k} \min\{p-1,2^{2-p}\}\,|\eta|^2. 
\end{align*}
Moreover, if $\theta \ge 1$, then there exists $C>0$ such that 
$||s|^{\theta-1}s| \le C|s|$ if $|s| \le \|v\|_{L^\infty(\Om)}$.
Thus, the weak Harnack 
inequality by Trudinger \cite[Theorem 1.2]{Tr} 
(see also Pucci and Serrin \cite[Theorem 7.1.2]{PS}) follows:
for any $\ol{B(x_0,4\rho)} \subset \Om$ and $\gamma \in (0,\frac{N}{N-2})\ (\gamma \in (0,\infty)$ if $N=2)$, 
there exists $C=C(N,\gamma,\Lambda_2/\ep,\rho,p,k,M_{p,k},m_{p,k})$
 such that
\begin{align}
\label{eq:harnack}
\rho^{-\frac{N}{\gamma}}\|v\|_{L^\gamma(B(x_0,2\rho))}
\le C \inf_{x \in B(x_0,2\rho)}v(x).
\end{align}

Suppose $v(x_0)=0$ with $x_0 \in \Om$. Then the set
$O=\{x \in \Om:v(x)=0\}$, which is closed relatively to $\Om$ since
$v$ is continuous, is nonempty. Since $v$ is continuous, 
if $x\in O$ and $\ol{B(x,4\delta)} \subset \Om$,
then $\inf_{B(x,2\rho)}{v}=v(x)=0$. From \eqref{eq:harnack} we have
that $\|v\|_{L^\gamma(B(x,2\rho))}=0$ so that $v \equiv 0$
in $B(x,2\rho)$.
So $O$ is also open and since $\Om$ is connected it must be $O=\Om$,
i.e., $v \equiv 0$ in $\Om$, which is a contradiction. Therefore, 
$v$ is strictly positive in $\Om$, i.e., $u_\ep<a$ in $\Om$.
\end{proof}

\section{Degenerate case}

In this section, we consider the case where $a(x)$ 
is constant in $\Om$. As introduced in Section 1,
this case has been already treated by several papers \cite{GS,GV,Gu,KV,KS}.
Our approach can be applied to the case.

Since $\nabla a \equiv 0$ in this case,
we note $\Phi_p(\nabla w,\nabla a)=\nabla_p w$ 
and Propositions \ref{prop:comparison}, \ref{prop:existence}
and Lemma \ref{lem:w_x} are all satisfied.  
However, Proposition \ref{prop:lcf} has to be changed
as follows.

\theoremstyle{plain}
\newtheorem{proplcf}{Proposition \ref{prop:lcf}'}
\renewcommand{\theproplcf}{}
\begin{proplcf}
\label{prop:deglcf'}
  Let $a(x)$ be a constant in $\Om$,
  and assume 
  $w$ to be the unique solution of \eqref{eq:lcf}.
  If $0<\theta<p-1$,
  then there exists $M>0$ independent of $\ep,\ \delta$ and $x_0$ 
  such that
  $w(x)=0$ for all $x \in 
  B(x_0,(1-M\delta^{(1+\theta)\gamma})^{1/\tau}\ep^{1/p})$,
  where
  \begin{align*}
  \gamma&=\frac{\frac{1}{1+\theta}-\frac1p}
	{N\left(\frac{1}{1+\theta}-\frac1p\right)+1}
	\in \left(0,\frac{1}{N+p^*}\right),\\
  \tau&=Np^*\left(\frac{1}{1+\theta}-\frac1p\right)+p^*
	\in \left(p^*,N+p^*\right),
  \end{align*}
  where $p^*=\frac{p}{p-1}$.
  In particular, $w(x_0)=0$ for arbitrary $\ep>0$ 
if $\delta^{(1+\theta)\gamma}<M^{-1}$.
\end{proplcf}

\begin{proof} 
 It is sufficient to prove
 the existence of dead core of solution of \eqref{eq:lcftw}.
 We define the diffusion
 and absorption energy functions $E_D(\rho)$ and $E_A(\rho)$ 
 in $(0,1)$
 as follows:
 \begin{align*}
  E_D(\rho)&= \int_{B_\rho} 
  |\nabla \tilde{w}(y)|^p\,dy,\\
  E_A(\rho)&=\int_{B_\rho}|\tilde{w}(y)|^{1+\theta}\,dy.
 \end{align*}
 The total energy function $E_T(\rho)$ is defined as
 \begin{equation*}
  E_T(\rho)=E_D(\rho)+\Lambda E_A(\rho).
 \end{equation*}
 The global total energy $E_T(1)$ is finite.
 Indeed, (we write $w$ instead 
 of $\tilde{w}$),
 multiplying the equation of \eqref{eq:lcftw} by 
 the nonnegative function $\delta-w 
 \in W^{1,p}_0(B_1)$ and 
 integrating by parts in $B_1$, we have 
 \begin{equation}
  \label{eq:degtotal}
   E_T(1) \le \Lambda\delta^{1+\theta}|B_1| 
   \le C\delta^{1+\theta}.
 \end{equation}

  Multiplying the equation of \eqref{eq:lcftw} by $w$ and 
 integrating by parts in $B_\rho$, we have also
 (now we shall write $S_\rho$ to represent $\partial B_\rho$)
 \begin{equation}
  \label{eq:degE_T1}
  E_T(\rho)=\int_{S_\rho}\nabla_p w(y)\cdot n\,w(y)\,ds,
 \end{equation}
 where $n=n(s)$ is the outward normal vector at $y \in S_\rho$.
 By \eqref{eq:degE_T1}
 \begin{align}
  \label{eq:degE_T}
  E_T(\rho) &=\int_{S_\rho}|\nabla_p w||w|\,ds
	\le \|\nabla w\|_{L^p(S_\rho)}^{p-1} \|w\|_{L^p(S_\rho)}.
 \end{align}
 On the other hand, by using spherical coordinates $(\omega,r)$ with
 center $x_0$, we have
 \begin{equation*}
  E_D(\rho)=\int_0^\rho \int_{S^{N-1}} |\nabla w(r\omega)|^p
	\, r^{N-1}\,d\omega\,dr.
 \end{equation*}
 Hence, $E_D$ is almost everywhere differentiable and
 \begin{align}
  \label{eq:degdE_D}
  \frac{dE_D(\rho)}{d\rho}
   & =\int_{S^{N-1}}|\nabla w(r\omega)|^p
	\rho^{N-1}\,d\omega
    =\int_{S_\rho} |\nabla w|^p\,ds.
 \end{align}
 Similarly,
\begin{align}
  \label{eq:degdE_A}
  \frac{dE_A(\rho)}{d\rho}
   =\int_{S_\rho} |w|^{1+\theta}\,ds.
 \end{align}
 Moreover, since $0<\theta<p-1$, 
 we have the following inequality (see 
 D\'{i}az \textit{et al}. \cite{DV,D,ADS}):
\begin{equation*}
  \|w\|_{L^p(S_\rho)} \le C\left(\|\nabla w\|_{L^p(B_\rho)}
  +\rho^{-\alpha}\|w\|_{L^{1+\theta}(B_\rho)}\right)^{\beta}
  \|w\|_{L^{1+\theta}(B_\rho)}^{1-\beta},
\end{equation*}
 where $C=C(N,\theta)$ and 
 \begin{align*}
  \alpha&=\frac{N(p-1-\theta)+p(1+\theta)}{p(1+\theta)}
	=N\left(\frac{1}{1+\theta}-\frac1p\right)+1
	\in \left(1,\frac{N}{p^*}+1\right) \subset (1,\infty),\\ 
   \beta&=\frac{N(p-1-\theta)+1+\theta}{N(p-1-\theta)+p(1+\theta)}
   =\frac{N\left(\frac{1}{1+\theta}-\frac1p\right)+\frac1p}
   {N\left(\frac{1}{1+\theta}-\frac1p\right)+1} 
	\in \left(\frac{1}{p},\frac{N+\frac{1}{p-1}}{N+p^*}\right) \subset (0,1).
 \end{align*}
 Thus,
 \begin{align}
  \label{eq:degE_A}
  \|w\|_{L^p(S_\rho)}^{1/\beta}
  &\le C\left(\|\nabla w\|_{L^p(B_\rho)}+\rho^{-\alpha}
  \|w\|_{L^{1+\theta}(B_\rho)}\right)\|w\|_{L^{1+\theta}(B_\rho)}
  ^{\frac{1-\beta}{\beta}} \notag \\
  &= C\left(\|\nabla
  w\|_{L^p(B_\rho)}\|w\|_{L^{1+\theta}(B_\rho)}
  ^{\frac{1-\beta}{\beta}}
  +\rho^{-\alpha}\|w\|_{L^{1+\theta}(B_\rho)}^{1/\beta}\right)
   \notag\\
  &= C\rho^{-\alpha}\left(\rho^{\alpha}
  E_D(\rho)^{\frac1p}
  E_A(\rho)^{\frac{1-\beta}{\beta(1+\theta)}}
  +E_A(\rho)^{\frac{1}{\beta (1+\theta)}}\right)  \notag\\
  & \le C\rho^{-\alpha}\left(
  E_T(\rho)^{\frac1p+\frac{1-\beta}{\beta(1+\theta)}}
  +E_A(1)^{\frac{1}{1+\theta}-\frac1p}
  E_A(\rho)^{\frac1p+\frac{1-\beta}{\beta(1+\theta)}}
  \right)  \notag\\
  & \le C\rho^{-\alpha} 
  E_T(\rho)^{\frac1p+\frac{1-\beta}{\beta(1+\theta)}}.
 \end{align}
 Here we have used 
 that $E_A(1) \le C \delta^{1+\theta}<C$ and $0<\theta<p-1$.
 Combining \eqref{eq:degE_T}--\eqref{eq:degdE_A} and \eqref{eq:degE_A},
 we obtain
 \begin{align*}
  E_T(\rho) & \le  C 
  \left(\frac{dE_T(\rho)}{d\rho}\right)^{(p-1)/p}
  \rho^{-\alpha \beta}
  E_T(\rho)^{\frac{\beta}{p}+\frac{1-\beta}{1+\theta}},
 \end{align*}
 that is,
 \begin{equation*}
   \frac{dE_T(\rho)}{d\rho} 
   \ge C \rho^{\tau-1}E_T(\rho)^{1-\gamma},
 \end{equation*}
 where
 \begin{align*}
  \gamma&=
	p^*(1-\beta)
   \left(\frac{1}{1+\theta}-\frac1p\right)
	=\frac{\frac{1}{1+\theta}-\frac1p}
	{N\left(\frac{1}{1+\theta}-\frac1p\right)+1}
	\in \left(0,\frac{1}{N+p^*}\right),\\
  \tau&=1+p^*\alpha\beta
	=Np^*\left(\frac{1}{1+\theta}-\frac1p\right)+p^*
	\in \left(p^*,N+p^*\right).
  \end{align*}
 Integrating it on $[\rho,1]$ 
 and using \eqref{eq:degtotal}, we have
 \begin{align*}
  E_T(\rho)^{\gamma} & \le E_T(1)^{\gamma}-
  C(1-\rho^{\tau})
  \le C(\rho^{\tau}-
  (1-M\delta^{(1+\theta) \gamma}))
 \end{align*}
 for some $M>0$,
 thus $E_T((1-M\delta^{(1+\theta) \gamma})^{1/\tau})=0$, i.e., 
 $\tilde{w}(y)=0$ for all $y \in 
 B(0,(1-M\delta^{(1+\theta) \gamma})^{1/\tau})$. 
 Scaling back to $x$, we conclude the assertion.
\end{proof}

As in Section 4, we obtain the corresponding 
Theorems \ref{thm:main'} and \ref{thm:main2'} below
to Theorems \ref{thm:main} and \ref{thm:main2}, respectively,
in the case when $a(x)$ is constant.  
For the proof of Theorem \ref{thm:main2'}, 
we have only to use the weak Harnack inequality
directly to $-\ep \Delta_p v+\Lambda_2 v^{\theta} \ge 0$ 
with $0<\theta<p-1$.
We note again that 
these have been already obtained by \cite{KV}. 

\begin{thm}
 \label{thm:main'}
 Assume $a(x)$ to be a positive constant. 
	Let $0<\theta<p-1$. 
 Then, there exist $L>0$ and $\ep_0 \in (0,\ep_a)$ 
 such that 
 for each $\ep \in (0,\ep_0)$
 the solution $u_{\ep}$ of \eqref{eq:p} satisfies
 $$u_{\ep}(x)=a(x)\quad \mbox{if}\ \dist(x,\partial \Om) \ge
	L{\ep}^{1/p}.$$
\end{thm}

\begin{thm}
 \label{thm:main2'}
 Assume $a(x)$ to be a positive constant.
 Let $\theta \ge p-1$.
 Then, for every $\ep \in (0,\ep_a)$, $u_\ep<a$ in $\Om$, 
and hence $\mathcal{O}_\ep =\emptyset$.
\end{thm}









\end{document}